\documentclass{amsart}
\usepackage{amsmath}
\usepackage{amssymb}
\usepackage{amscd}
\usepackage{color}
\usepackage[all]{xy}

\newcommand{\fm}{\mathfrak{m}}

\newcommand{\Hom}{\operatorname{Hom}}
\newcommand{\Tor}{\operatorname{Tor}}
\newcommand{\im}{\operatorname{im}}

\newcommand{\Ann}{\operatorname{Ann}}

\newcommand{\fd}{\operatorname{fd}}

\newtheorem{theorem}{Theorem}[section]
\newtheorem{lemma}[theorem]{Lemma}
\newtheorem{corollary}[theorem]{Corollary}
\newtheorem{proposition}[theorem]{Proposition}

\newtheorem{example}[theorem]{Example}
\newtheorem{remark}[theorem]{Remark}

\begin{document}

\title[A characterization of monoid graded semihereditary rings]
{A characterization of monoid graded semihereditary rings}
\author[P. Sahandi and N. Shirmohammadi]
{Parviz Sahandi and Nematollah Shirmohammadi}

\address{(Sahandi) Department of Pure Mathematics, Faculty of Mathematics, Statistics and Computer Science,  University of Tabriz, Tabriz, Iran.} \email{sahandi@ipm.ir, sahandi@tabrizu.ac.ir}
\address{(Shirmohammadi) Department of Pure Mathematics, Faculty of Mathematics, Statistics and Computer Science,  University of Tabriz, Tabriz, Iran.} \email{shirmohammadi@tabrizu.ac.ir}



\begin{abstract}
Let  $\Gamma$ be a cancellation monoid and $R=\bigoplus_{\alpha \in \Gamma}R_{\alpha}$ be a $\Gamma$-graded ring. It is shown that $R$ is graded left semihereditary if and only if $R$ is graded left coherent and every graded submodule of a flat left $R$-module is flat. Hence it gives a new characterization of graded-Pr\"{u}fer domains.
\end{abstract}

\maketitle

\section{Introduction and preliminaries}

Let $\Gamma$ be a cancelation monoid, and $R=\bigoplus_{\alpha \in \Gamma}R_{\alpha}$ be a $\Gamma$-graded ring with identity. The authors in \cite{ss25} defined and studied the notion of graded left (right) semihereditary rings which is a generalization of graded-Pr\"{u}fer domains. Recall that $R$ is a \emph{graded left (right) semihereditary ring} if each finitely generated homogeneous left (right) ideal of $R$ is a projective $R$-module. They gave some characterizations of these kind of graded rings (see \cite[Theorems 5.4, 5.8 and Proposition 5.5]{ss25}). In particular, it was shown that $R$ is graded left semihereditary if and only if $R$ is graded left coherent and every homogeneous left ideal is flat \cite[Proposition 5.5]{ss25}. Recall that a graded ring $R$ is a \emph{graded left coherent ring} if every finitely generated graded left ideal is finitely presented. This paper aims to obtain a generalization of this characterization. In fact, we show that the $\Gamma$-graded ring $R$ is graded left semihereditary if and only if $R$ is graded left coherent and every graded submodule of a flat left $R$-module is flat. To prove this result, we need to study the graded version of the notion of weak dimension of a ring.

Let $\Gamma$ be a cancelation monoid, that is, $\Gamma$ is a semigroup with a neutral element $\varepsilon$, not necessarily commutative, and
$\Gamma$ satisfies the (left and right) cancelation law: $\alpha\alpha_1=\alpha\alpha_2$ implies $\alpha_1=\alpha_2$, and $\alpha_1\alpha=\alpha_2\alpha$ implies $\alpha_1=\alpha_2$ for all $\alpha,\alpha_1,\alpha_2\in\Gamma$. By a \emph{$\Gamma$-graded ring} $R$, we mean an associative ring with the multiplicative identity 1, which has a direct sum decomposition $R=\bigoplus_{\alpha \in \Gamma}R_{\alpha}$, where each $R_{\alpha}$ is an additive subgroup of $R$ such that $R_{\alpha}R_{\beta}\subseteq R_{\alpha\beta}$ for all $\alpha,\beta\in\Gamma$. For each $\alpha\in\Gamma$, we call $0\neq a\in R_{\alpha}$ a homogeneous element of degree $\alpha$ and write $\deg(a)=\alpha$. For the basics about graded rings and
graded modules we used in this paper, one may refer to \cite{no04}, though it mainly deals with group $G$-graded rings.

Let $R=\bigoplus_{\alpha \in \Gamma}R_{\alpha}$ be a $\Gamma$-graded ring as above. By the definition, it is clear that the degree-$\varepsilon$ part $R_{\varepsilon}$ is a subring of $R$, and by using the cancelation law on $\Gamma$, it is easy to check that $1\in R_{\varepsilon}$. Let $I$ be a left (resp. right) ideal of $R$. If $I$ is generated by homogeneous elements, then we say that $I$ is a \emph{homogeneous left (resp. right) ideal} of $R$. A homogeneous left ideal of $R$ is called a {\em maximal homogeneous left ideal} if it is maximal among proper homogeneous left ideals of $R$. It is easy to see that each proper homogeneous left ideal of $R$ is contained in a maximal homogeneous left ideal of $R$.

Let $R=\bigoplus_{\alpha \in \Gamma}R_{\alpha}$ be a $\Gamma$-graded ring. By a $\Gamma$-graded $R$-module $M$, we mean a left $R$-module which has a direct sum
decomposition $M=\bigoplus_{\alpha \in \Gamma}M_{\alpha}$, where each $M_{\alpha}$ is an additive subgroup of $M$, such that $R_{\alpha}M_{\beta}\subseteq M_{\alpha\beta}$ for all $\alpha,\beta\in\Gamma$. Similarly, one can define a $\Gamma$-graded right $R$-module. For each $\alpha\in\Gamma$, we call $0\neq x\in M_{\alpha}$ a \emph{homogeneous element} of degree $\alpha$ and write $\deg(x)=\alpha$.
Let $N$ be a submodule of $M$. If $N$ is generated by some homogeneous elements, then we say that \emph{$N$ is a graded submodule of $M$}.

Let $M=\bigoplus_{\alpha \in \Gamma}M_{\alpha}$ and $N=\bigoplus_{\alpha \in \Gamma}N_{\alpha}$ be $\Gamma$-graded left $R$-modules and $\alpha\in\Gamma$. A \emph{homogeneous $R$-module homomorphism of degree $\alpha$} is an $R$-module homomorphism $f:M\to N$, such that $f(M_{\beta})\subseteq N_{\beta\alpha}$ for any $\beta\in\Gamma$. We say that $f:M\to N$ is a \emph{homogeneous $R$-module homomorphism} if $\alpha=\varepsilon$. In this case, it is easy to see that $\ker(f)$ and $\im(f)$ are graded submodules of $M$ and $N$, respectively. Let $\Hom_R(M,N)$ be the group of all $R$-module homomorphism from $M$ to $N$, and denote by $\Hom_{R,\alpha}(M,N)$ the subgroup of $\Hom_R(M,N)$ consisting of all homogeneous $R$-module homomorphisms of degree $\alpha$. As in \cite{no82} and \cite{ss25}, we put $$\sideset{^*}{_{R}}{\Hom}(M,N):=\bigoplus_{\alpha\in\Gamma}\Hom_{R,\alpha}(M,N).$$ One notices that the functors $\sideset{^*}{_{R}}{\Hom}(M,-)$ and $\sideset{^*}{_{R}}{\Hom}(-,N)$ are left exact on the category of $\Gamma$-graded left $R$-modules.

Let $M$ be a $\Gamma$-graded right $R$-module and $N$ be a $\Gamma$-graded left $R$-module. As in \cite{no82} and \cite{ss25}, the tensor product $M\otimes_RN$ is a $\Gamma$-graded $R$-module with $(M\otimes_R N)_{\alpha}$ generated over $R_{\varepsilon}$ by the tensor products $m\otimes n$ of homogeneous elements $m\in H(M)$, $n\in H(N)$ with $\deg(m)+\deg(n)=\alpha$.

\section{Main results}

Let $\Gamma$ be a cancelation monoid with neutral element $\varepsilon$ and $R=\bigoplus_{\alpha \in \Gamma}R_{\alpha}$ a $\Gamma$-graded ring. We begin with the following graded version of adjoint isomorphism which is stated in \cite[Proposition 1.2.14]{no82} in the case where $\Gamma$ is a group.

\begin{proposition}(Graded adjoint isomorphism)\label{adj}
Assume that $R=\bigoplus_{\alpha \in \Gamma}R_{\alpha}$ and $S=\bigoplus_{\alpha \in \Gamma}S_{\alpha}$ are $\Gamma$-graded rings. Assume that $M$, $N$, $L$ are $\Gamma$-graded right $R$-module, $R-S$-bimodule, and right $S$-module, respectively. Then there is a natural isomorphism
$$\tau_{N,M,L}:\sideset{^*}{_{S}}{\Hom}(N\otimes_RM,L)\to\sideset{^*}{_{R}}{\Hom}(N,\sideset{^*}{_{S}}{\Hom}(M,L)),$$ such that $\tau_{N,M,L}(\sideset{^*}{_{S,\alpha}}{\Hom}(N\otimes_RM,L))=\sideset{^*}{_{R,\alpha}}{\Hom}(N,\sideset{^*}{_{S}}{\Hom}(M,L))$ for $\alpha\in\Gamma$.
\end{proposition}
\begin{proof}
For $f\in\sideset{^*}{_{S,\alpha}}{\Hom}(N\otimes_RM,L)$, $\alpha\in\Gamma$, let $\tau_{N,M,L}:f\mapsto\tau(f)$ with $\tau(f)(x)(y)=f(x\otimes y)$ where $x\in N$ and $y\in M$. It is the straightforward graded version of the classical argument in the classical situation that $\tau_{N,M,L}$ is a well defined graded isomorphism \cite[Theorem 2.75]{ro09}.
\end{proof}

The notion of graded injective modules was investigated in \cite{ff74, bh98} in the case where $\Gamma=\mathbb{Z}$, in \cite{no04} in the case where $\Gamma$ is a group, and in \cite{ss25} in the general case of cancelation monoids. Recall that a $\Gamma$-graded left $R$-module $E$ is a \emph{graded injective module} (for short, \emph{gr-injective module}) if $\sideset{^*}{_{R}}{\Hom}(-,E)$ is an exact functor in the category of $\Gamma$-graded left $R$-modules. In \cite[Theorem 3.7]{ss25}, it is shown that every $\Gamma$-graded left $R$-module $M$ can be embedded in a graded injective envelope $\mathrm{E}^{gr}(M)$. A $\Gamma$-graded left $R$-module $E$ is said to be a \emph{graded injective cogenerator} if $E$ is gr-injective, and for every $\Gamma$-graded left $R$-module $M$ and every non-zero (homogeneous) element $x$ of $M$, there is a homogeneous homomorphism $\phi:M\to E$ such that $\phi(x)\neq0$.

Let $\{M_{j}\}_{j\in J}$ be a family of $\Gamma$-graded left $R$-modules. The direct sum $\bigoplus_{j}M_j$ exists in the category of $\Gamma$-graded left $R$-modules with $(\bigoplus_{j}M_j)_{\alpha}=\bigoplus_{j}(M_j)_{\alpha}$ for all $\alpha\in\Gamma$. Recall that a $\Gamma$-graded left $R$-module $F=\bigoplus_{\alpha\in\Gamma}F_{\alpha}$ is called \emph{graded free} (for short, \emph{gr-free}) if $F$ is a free $R$-module (in the usual sense) but with a homogeneous free $R$-basis; namely, $F$ has a free $R$-basis $\{e_i\}_{i\in J}$ consisting of homogeneous elements.

\begin{example}\label{}
Assume that $R=\bigoplus_{\alpha \in \Gamma}R_{\alpha}$ is a $\Gamma$-graded ring. One notices that, for every $\alpha \in \Gamma$, there is a gr-free $R$-module with homogeneous basis $e_{\alpha}$ of degree $\alpha$, $Re_{\alpha}$ \cite[Page 2698]{li12}. Then $\widetilde{\mathbb{E}}:=\mathrm{E}^{gr}(\bigoplus_{\alpha, \fm}Re_{\alpha}/\fm e_{\alpha})$ is a graded injective cogenerator of $R$, in which $\alpha$ varies in $\Gamma$ and $\fm$ varies in the set of all maximal homogeneous left ideals of $R$. Indeed, let $M$ be a nonzero $\Gamma$-graded left $R$-module and $0\neq x\in M$ is a homogeneous element of degree $\alpha$. It can be see that $f_{\alpha}:Rx\to Re_{\alpha}/\Ann(x)e_{\alpha}$ defined by $f_{\alpha}(rx)=re_{\alpha}+\Ann(x)e_{\alpha}$ is a homogeneous isomorphism. Choose now a maximal homogeneous left ideal $\fm$ of $R$ containing $\Ann(x)$ and consider the natural homogeneous epimorphism $g_{\alpha}:Re_{\alpha}/\Ann(x)e_{\alpha}\to Re_{\alpha}/\fm e_{\alpha}$. Then the composition $g_{\alpha}f_{\alpha}:Rx\to Re_{\alpha}/\fm e_{\alpha}$ is a homogeneous homomorphism such that $(g_{\alpha}f_{\alpha})(x)\neq0$. Consider the homogeneous embedding $Re_{\alpha}/\fm e_{\alpha}\hookrightarrow\widetilde{\mathbb{E}}$. Hence we have constructed a homogeneous homomorphism $h:Rx\to \widetilde{\mathbb{E}}$ such that $h(x)\neq0$. This homomorphism extends to a homogeneous homomorphism $\varphi:M\to \widetilde{\mathbb{E}}$ with $\varphi(x)=h(x)\neq0$ by the graded injectivity of $\widetilde{\mathbb{E}}$.
\end{example}

Let $\mathbb{E}$ be a $\Gamma$-graded injective cogenerator left $R$-module. For each $\Gamma$-graded right $R$-module $M$, set $M':=\sideset{^*}{_{R}}{\Hom}(M,\mathbb{E})$, which is a $\Gamma$-graded left $R$-module, and for a homogeneous homomorphism $\varphi:M\to N$ of $\Gamma$-graded right $R$-modules, the induced homogeneous homomorphism $\varphi':N'\to M'$ is defined by $\varphi'(f):=f\varphi$ for $f\in\Hom_{R,\alpha}(N,\mathbb{E})$, $\alpha\in\Gamma$.

\begin{lemma}\label{exact}
Assume that $R=\bigoplus_{\alpha \in \Gamma}R_{\alpha}$ is a $\Gamma$-graded ring and $\mathbb{E}$ is a $\Gamma$-graded injective cogenerator left $R$-module. Then a sequence $M\stackrel{\varphi}{\to} N\stackrel{\psi}{\to} L$ of $\Gamma$-graded right $R$-modules and homogeneous homomorphisms is exact if and only if the induced sequence $L'\stackrel{\psi'}{\to} N'\stackrel{\varphi'}{\to} M'$ is exact.
\end{lemma}
\begin{proof}
The argument is a straightforward graded version of the classical argument, see \cite[Proposition 4.8, Page 124]{la98}.
\end{proof}

Recall that a $\Gamma$-graded left $R$-module $F$ is a \emph{graded-flat module} (for short, \emph{gr-flat module}) if $f\otimes_R1_F:M\otimes_RF\to N\otimes_RF$ is a homogeneous monomorphism for every homogeneous monomorphism $f:M\to N$ of $\Gamma$-graded right $R$-modules $M, N$. It is known that $F$ is gr-flat if and only if it is flat \cite[Corollary 3.13]{ss25}.

\begin{proposition}\label{gradj}
Assume that $R=\bigoplus_{\alpha \in \Gamma}R_{\alpha}$ is a $\Gamma$-graded ring and $\mathbb{E}$ is a $\Gamma$-graded injective cogenerator left $R$-module. Then a $\Gamma$-graded right $R$-module $M$ is gr-flat if and only if $M'$ is a gr-injective left $R$-module.
\end{proposition}
\begin{proof}
Assume that $M$ is a $\Gamma$-graded right gr-flat $R$-module. Using the adjoint isomorphism, it can be seen that $\sideset{^*}{_{R}}{\Hom}(-,M')$ is an exact functor. Thus $M'$ is a gr-injective left $R$-module.

Conversely, assume that $M$ is a $\Gamma$-graded right $R$-module such that $M'$ is a gr-injective left $R$-module. Let $N_1\to N_2$ be a homogeneous monomorphism of $\Gamma$-graded left $R$-modules. By Proposition \ref{adj} the following diagram commutes:
\begin{displaymath}
\xymatrix{ \sideset{^*}{_{R}}{\Hom}(N_2,M') \ar[r] \ar[d]_{\cong} &
\sideset{^*}{_{R}}{\Hom}(N_1,M') \ar[d]^{\cong}   \\
(N_2\otimes_RM)' \ar[r] & (N_1\otimes_RM)'. }
\end{displaymath}
Since the bottom homomorphism is an epimorphism, Lemma \ref{exact} shows that $0\to N_1\otimes_RM\to N_2\otimes_RM$ is exact. Hence $M$ is gr-flat.
\end{proof}

Assume that $M$ is a $\Gamma$-graded left (resp. right) $R$-module. A \emph{graded flat resolution} of $M$ is an exact sequence $$\cdots\to F_2\to F_1\to F_0\to M\to 0,$$ in which each $F_i$ is a gr-flat left (resp. right) $R$-module with homogeneous homomorphisms. We say that gr-$\fd_R(M)\leq n$ (gr-$\fd_R(M)$ abbreviates \emph{graded flat dimension}) if there is a finite graded flat resolution $$0\to F_n\to\cdots\to F_1\to F_0\to M\to 0.$$ If no such finite resolution exists, then we write gr-$\fd_R(M)=\infty$. Otherwise, define gr-$\fd_R(M)=n$ if $n$ is the length of a shortest graded flat resolution of $M$.

\begin{remark}\label{}{\em
Assume that $R=\bigoplus_{\alpha \in \Gamma}R_{\alpha}$ is a $\Gamma$-graded ring and $M$ is a $\Gamma$-graded left $R$-module. Then $$\text{gr-}\fd_R(M)=\fd_R(M),$$ where $\fd_R(M)$ is the usual flat dimension of $M$. Indeed, since gr-flat $R$-modules are flat \cite[Corollary 3.13]{ss25}, we have $\fd_R(M)\leq$gr-$\fd_R(M)$. Assume now that $t:=\fd_R(M)$ is finite and consider a graded flat resolution $$\cdots\to F_1\to F_0\to M\to 0$$ of $M$. Let $K_t:=\ker(F_t\to F_{t-1})$ which is a $\Gamma$-graded left $R$-module. Then we obtain an exact sequence $$0\to K_t\to F_{t-1}\to\cdots\to F_1\to F_0\to M\to 0.$$ Since $\fd_R(M)=t$, $K_t$ is a flat $R$-module by \cite[Proposition 8.17]{ro09}. Hence $K_t$ is a gr-flat $R$-module. It follows that gr-$\fd_R(M)\leq\fd_R(M)$.}
\end{remark}

The \emph{graded left weak dimension} of a $\Gamma$-graded ring $R$ is defined by
$$\text{gr-}\ell\text{wD}(R):=\sup\{\fd_R(M)\mid M\text{ is a }\Gamma\text{-graded left }R\text{-module}\}.$$
The \emph{graded right weak dimension} of $R$ is defined similarly and denoted by $\text{gr-rwD}(R)$.

\begin{proposition}\label{}
Assume that $R=\bigoplus_{\alpha \in \Gamma}R_{\alpha}$ is a $\Gamma$-graded ring. Then $$\text{gr-}\ell\text{wD}(R)=\text{gr-rwD}(R).$$
\end{proposition}
\begin{proof}
By the same argument as in \cite[Proposition 8.17]{ro09}, for a non-negative integer $n$, one can show that for a $\Gamma$-graded right $R$-module $M$, $\text{gr-}\fd_R(M)\leq n$ if and only if $\Tor^R_{n+1}(M,N)=0$ for all $\Gamma$-graded left $R$-module $N$. It results that $\text{gr-}\ell\text{wD}(R)\leq n$ if and only if $\Tor^R_{n+1}(M,N)=0$ for all $\Gamma$-graded right $R$-module $M$ and $\Gamma$-graded left $R$-module $N$. Similarly, one has $\text{gr-}r\text{wD}(R)\leq n$ if and only if $\Tor^R_{n+1}(M,N)=0$ for all $\Gamma$-graded right $R$-module $M$ and $\Gamma$-graded left $R$-module $N$. Therefore $\text{gr-}\ell\text{wD}(R)=\text{gr-rwD}(R)$.
\end{proof}

The \emph{graded weak dimension} gr-$\text{wD}(R)$ of a $\Gamma$-graded ring $R$ is the common value of $\text{gr-}\ell\text{wD}(R)$ and $\text{gr-rwD}(R)$. In general, we have gr-$\text{wD}(R)\leq\text{wD}(R)$, where $\text{wD}(R):=\sup\{\fd_R(M)\mid M\text{ is a left }R\text{-module}\}$ is the usual weak dimension of $R$ and the equality can not be hold in general, see Remark \ref{rem}.

\begin{lemma}\label{grtor}
Assume that $R=\bigoplus_{\alpha \in \Gamma}R_{\alpha}$ is a $\Gamma$-graded ring and $N$ is a $\Gamma$-graded left $R$-module. Then $N$ is gr-flat if and only if $\Tor^R_1(R/I,N)=0$ for every homogeneous right ideal $I$ of $R$.
\end{lemma}
\begin{proof}
One implication is easy. For the other one, assume that $\Tor^R_1(R/I,N)=0$ for every homogeneous right ideal $I$ of $R$. Thus $0\to I\otimes_RN\to R\otimes_RN$ is exact for every homogeneous right ideal $I$ of $R$. It follows from Lemma \ref{exact} that $(R\otimes_RN)'\to (I\otimes_RN)'\to 0$ is exact for every homogeneous right ideal $I$ of $R$. Hence by Proposition \ref{adj}, the sequence $\sideset{^*}{_{R}}{\Hom}(R,N')\to\sideset{^*}{_{R}}{\Hom}(I,N')\to0$ is exact for every homogeneous right ideal $I$ of $R$. Therefore $N'$ is gr-injective by \cite[Theorem 3.2]{ss25}. This shows that $N$ is gr-flat by Proposition \ref{gradj}.
\end{proof}

\begin{theorem}\label{grwD}
Assume that $R=\bigoplus_{\alpha \in \Gamma}R_{\alpha}$ is a $\Gamma$-graded ring. Then
\begin{align*}
\text{gr-wD}(R)=&\sup\{\fd_R(R/I)\mid I \text{ is a homogeneous right ideal}\}\\
=&\sup\{\fd_R(R/J)\mid J \text{ is a homogeneous left ideal}\}.
\end{align*}
\end{theorem}
\begin{proof}
There is nothing to prove if $\sup\{\fd_R(R/I)\mid I \text{ is a homogeneous right ideal}\}=\infty$. Assume that $\fd_R(R/I)\leq n$ for all homogeneous right ideal $I$. Thus $\Tor^R_{n+1}(R/I,N)=0$ for every $\Gamma$-graded left $R$-module $N$. Now fix a $\Gamma$-graded left $R$-module $N$ and take a graded flat resolution of $N$ with $n$th syzygy $K_n$. Then $\Tor^R_{1}(R/I,K_n)\cong\Tor^R_{n+1}(R/I,N)=0$ for all homogeneous right ideal $I$. Hence $K_n$ is gr-flat by Lemma \ref{grtor}. This shows that $\fd_R(N)\leq n$. Therefore the first equality holds. The second equality similarly obtains.
\end{proof}

\begin{corollary}\label{8.26}
Assume that $R=\bigoplus_{\alpha \in \Gamma}R_{\alpha}$ is a $\Gamma$-graded ring. Then every homogeneous left ideal of $R$ is gr-flat if and only if every graded submodule of a gr-flat left $R$-module is gr-flat.
\end{corollary}
\begin{proof}
One implication is clear. To prove the other, assume that every homogeneous left ideal of $R$ is gr-flat. Then $0\to I\to R\to R/I\to0$ is a gr-flat resolution of $R/I$ for each homogeneous left ideal $I$. This shows that $\fd_R(R/I)=\text{gr-}\fd_R(R/I)\leq1$ for each homogeneous left ideal $I$; so that $\text{gr-wD}(R)\leq1$ by Theorem \ref{grwD}. Assume now that $N$ is a graded submodule of a gr-flat left $R$-module $M$. Then $\fd_R(M/N)\leq1$. Extend the natural projection $M\to M/N\to0$ to a gr-flat resolution of $M/N$. By \cite[Proposition 8.17]{ro09}, the 0th syzygy of this resolution is (gr-)flat. But the 0th syzygy is $N$. Hence $N$ is gr-flat.
\end{proof}

Recall that a graded ring $R$ is a \emph{graded left (right) coherent ring} if every finitely generated graded left (right) ideal is finitely presented \cite{ss25}. In \cite[Theorem 4.1]{c60}, Chase showed that a ring $R$ is left semihereditary if and only if $R$ is left coherent and every submodule of a flat left $R$-module is flat. We shall now state and prove the principal result of this paper which is a graded analogue of Chase's result.

\begin{theorem}\label{main}
Assume that $R=\bigoplus_{\alpha \in \Gamma}R_{\alpha}$ is a $\Gamma$-graded ring. Then $R$ is graded left semihereditary if and only if $R$ is graded left coherent and every graded submodule of a gr-flat left $R$-module is gr-flat.
\end{theorem}
\begin{proof}
Assume first that $R$ is graded left semihereditary. Then every finitely generated homogeneous left ideal $I$ is (gr-)projective; hence, $I$ is finitely presented, by \cite[Proposition 3.11]{ro09}. Therefore $R$ is graded left coherent. Since every finitely generated homogeneous left ideal is (gr-)projective, it is (gr-)flat. It follows from \cite[Remark 3.14(1)]{ss25} that every graded left ideal is gr-flat. Hence every graded submodule of a gr-flat left $R$-module is gr-flat by Corollary \ref{8.26}.

Conversely, if $I$ is a finitely generated homogeneous left ideal, then $I$ is a graded submodule of the gr-flat module $R$, and so $I$ is gr-flat; since $R$ is left coherent, $I$ is also finitely presented. Hence, $I$ is projective, by \cite[Remark 3.14(2)]{ss25}, and so $R$ is graded left semihereditary.
\end{proof}

Similarly one can prove the above statement by replacing ``right'' instead of ``left''.

\begin{corollary}\label{wD}
Assume that $R=\bigoplus_{\alpha \in \Gamma}R_{\alpha}$ is a $\Gamma$-graded ring. Then $R$ is graded left (right) semihereditary if and only if $R$ is graded left (right) coherent and $\text{gr-wD}(R)\leq1$.
\end{corollary}

Recall that a $\Gamma$-graded commutative integral domain $R$ is a \emph{graded-Pr\"{u}fer domain} if each nonzero finitely generated homogeneous ideal of $R$ is invertible, equivalently projective.

\begin{corollary}\label{}
Assume that $R=\bigoplus_{\alpha \in \Gamma}R_{\alpha}$ is a $\Gamma$-graded integral domain. Then $R$ is graded-Pr\"{u}fer if and only if $R$ is graded left coherent and $\text{gr-wD}(R)\leq1$.
\end{corollary}

\begin{remark}\label{rem}{\em
Assume $R=\bigoplus_{\alpha \in \Gamma}R_{\alpha}$ is a $\Gamma$-graded ring. Then gr-$\text{wD}(R)\leq\text{wD}(R)$. We note that the equality can not be hold in general. Let $D$ be a Dedekind domain which is not a field and set $R:=D[X,X^{-1}]=D[X]_X$, for an indeterminate $X$ over $D$. Then by \cite[Example 3.6]{ac13}, $R$ is a $\mathbb{Z}$-graded-Pr\"{u}fer domain which is not Pr\"{u}fer. So that $\text{gr-wD}(R)\leq1$ by Corollary \ref{wD}. Note that $R$ is a coherent ring since it is Noetherian. Thus $\text{wD}(R)>1$ by \cite[Theorem 4.32]{ro09}. Then gr-$\text{wD}(R)\neq\text{wD}(R)$.
}
\end{remark}

Using Corollary \ref{wD} we can give an example of a graded left coherent ring but not a graded right coherent.

\begin{example}\label{}{\em
We can give a graduation to Chase's example, which provides us a graded left coherent ring but not a graded right coherent. Let $S$ be a von Neumann regular and nonsemisimple ring. Hence there is an ideal $I$ of $S$ such that, as a submodule of the right $S$-modules $S$, $I$ is not a direct summand. Let $R:=S/I$, which is also a von Neumann regular ring. The triangular ring
$$T:=\left(\begin{array}{cc}
                        R & R \\
                        0 & S
\end{array}\right)$$ is an $\mathbb{N}_0$-graded ring, with
$T_0=\left(\begin{array}{cc}
                        R & 0 \\
                        0 & S
\end{array}\right)$,
$T_1=\left(\begin{array}{cc}
                        0 & R \\
                        0 & 0
\end{array}\right)$
and $T_i=0$ for $i>1$. According to \cite[Example 5.2(3)]{ss25}, $T$ is graded left semihereditary but it is not a graded right semihereditary ring. Therefore $T$ is a graded left coherent ring but not a graded right coherent.
}
\end{example}



\begin{thebibliography}{10}

\bibitem{ss25} H. F. G. Al-Kharsan, P. Sahandi and N. Shirmohammadi, {\em On monoid graded semihereditary rings}, Commun. Alg. to appear, (arXiv:2603.20202v1, 24 Feb. 2026).


\bibitem{ac13} D. F. Anderson and G. W. Chang, {\em Graded integral domains and Nagata rings},
J. Algebra {\bf 387}, (2013), 169--184.




\bibitem{bh98} W. Bruns and J. Herzog, {\em Cohen-Macaulay Rings}, Revised Edition, Cambidge Studies in Advanced Mathematics {\bf 39}, 1998.



\bibitem{c60} S. U. Chase, {\em Direct products of modules}, Trans. Amer. Math. Soc. {\bf 97}, (1960), 457--473.


\bibitem{ff74} R. Fossum and H. B. Foxby, {\em The category of graded modules}, Math. Scand. {\bf 35}, No. 2 (1974), 288--300.





\bibitem{la98} T. Y. Lam, {\em Lectures on Modules and Rings}, Springer-Verlag, New York, 1998.

\bibitem{li12} H. Li, {\em On monoid graded local rings}, J. Pure Appl. Alg., {\bf 216}, (2012), 2697--2708.

\bibitem{no82} C. N\u{a}st\u{a}sescu and F. V. Oystaeyen, {\em Graded Ring Theory}, in: Math. Library, {\bf 28}, North Holland, Amsterdam, 1982.

\bibitem{no04} C. N\u{a}st\u{a}sescu and F. V. Oystaeyen, {\em Methods of Graded Rings}, Lecture Notes in Mathematics, {\bf 1836}, Springer-Verlag,  Berlin, 2004.



\bibitem{ro09} J. J. Rotman, {\em An Introduction to Homological Algebra}, Springer-Verlag, New York, 2009.



\end{thebibliography}
\end{document}